\documentclass[12pt]{amsart}
\usepackage{ae,aecompl}

\usepackage[T1]{fontenc}
\usepackage[latin9]{inputenc}

\usepackage[letterpaper]{geometry}
\usepackage{amsthm}
\usepackage{amssymb}

\numberwithin{equation}{section}
\numberwithin{figure}{section}
\theoremstyle{plain}
\newtheorem{thm}{Theorem}[section]
\theoremstyle{plain}
\newtheorem{lem}[thm]{Lemma}
\theoremstyle{definition}
\newtheorem{example}[thm]{Example}

\begin{document}

\title{Noncommutative Semialgebraic Sets in Nilpotent Variables}

\author{Terry A. Loring}

\author{Tatiana Shulman}

\address{Department of Mathematics and Statistics, University of New Mexico,
Albuquerque, NM 87131, USA.}

\address{Department of Mathematics, University of Copenhagen, Universitetsparken
5, DK-2100 Copenhagen \O , Denmark}



\begin{abstract}
We solve the lifting problem in $C^{*}$-algebras for many sets of
relations that include the relations $x_{j}^{N_{j}}=0$ on each variable.
The remaining relations must be of the form
$\left\Vert p(x_{1},\ldots,x_{n})\right\Vert \leq C$
for $C$ a positive constant and $p$ a noncommutative $*$-polynomial
that is in some sense homogeneous. For example, we prove liftability
for the set of relations  
\[
x^{3}=0,\ y^{4}=0,\ z^{5}=0,\ xx^{*}+yy^{*}+zz^{*}\leq1.
\]
Thus we find more noncommutative semialgebraic sets that have the
topology of noncommutative absolute retracts.
\end{abstract}
\maketitle

\section{Introduction}

Lifting problems involving norms and star-polynomials are fundamental
in $C^{*}$-algebras. They arise in basic lemmas in the subject, as
we shall see in a moment. They also arise in descriptions of the boundary
map in $K$-theory, in technical lemmas on inductive limits, and have
of course been around in operator theory. Much of our understanding
of the Calkin algebra comes from having found properties of its cosets
that exist only when some operator in a coset has that property.

Let $A$ denote a $C^{*}$-algebra and let $I$ be an ideal in $A.$
The quotient map will be denoted $\pi:A\rightarrow A/I.$ Of course
$A/I$ is a $C^{*}$-algebra, but let us ponder how we know this.
The standard proof uses an approximate unit $u_{\lambda}$ and an
approximate lifting property. The lemma used is that for any
approximate unit $u_{\lambda},$ and any $a$ in $A,$ 
\[
\lim_{\lambda}\left\Vert a(1-u_{\lambda})\right\Vert 
= \left\Vert \pi(a)\right\Vert 
\]
and trivially we obtain as a corollary
\[
\lim_{\lambda}\left\Vert (1-u_{\lambda})b(1-u_{\lambda})\right\Vert 
= \left\Vert \pi(b)\right\Vert .
\]
For a large $\lambda,$ the lift $\bar{x}=a(1-u_{\lambda})$ of $\pi(a)$
approximately achieves \emph{two} norm conditions,
\[
\left\Vert \bar{x}\right\Vert 
\approx
\left\Vert \pi(\bar{x})\right\Vert ,\quad\left\Vert \bar{x}^{*}\bar{x}\right\Vert
\approx
\left\Vert \pi(\bar{x})^{*}\pi(\bar{x})\right\Vert .
\]
The equality
$\left\Vert \bar{x}\right\Vert ^{2}
=\left\Vert \bar{x}^{*}\bar{x}\right\Vert $
upstairs now passes downstairs, so $A/I$ is a $C^{*}$-algebra.

We have an eye on potential applications in noncommutative real algebraic
geometry \cite{HeltonPositivstellensatz,lasserrePutinarSemiAlgebraic}.
What essential differences are there between real algebraic geometry
and noncommutative real algebraic geometry? Occam would cut between
these fields with the equation 
\[
x^{n}=0.
\]
Could we just exclude this equation? Probably not. A search of the
physics literature finds that polynomials in nilpotent variables are
gaining popularity. Two examples to see are \cite{efetov1999supersymmetry}
in condensed matter physics, and \cite{mandilara2006quantum} in quantum
information.

Focusing back on lifting problems, we recall what is known about lifting
nilpotents up from general $C^{*}$-algebra quotients. Akemann and
Pedersen \cite{AkePederIdealPerturb} showed the relation $x^{2}=0$
lifts, and Olsen and Pedersen \cite{OlsenPedLifting} did the same
for $x^{n}=0.$ Akemann and Pedersen \cite{AkePederIdealPerturb}
also showed that if $x^{n-1}\neq0$ for some $x\in A/I$ then one
can find a lift $X$ of $x$ with 
\[
\left\Vert X^{j}\right\Vert 
= \left\Vert x^{j}\right\Vert ,\quad(j=1,\ldots,n-1).
\]
If $x^{n}=0$ and $x^{n-1}\neq0$ then we would like to combine these
results, lifting both the nilpotent condition and the $n-1$ norm
conditions. It was not until recently, in \cite{Shulman_nilpotents},
that it was shown one could lift just the two relations
\[
\left\Vert x\right\Vert \leq C,\quad x^{n}=0
\]
for $C>0.$ 

Here we show how to lift a nilpotent and all these norm conditions,
so show the liftablity of the set of relations
\[
\left\Vert x^{j}\right\Vert \leq C_{j},\quad j=1,\ldots,n,
\]
even if $C_{n}=0.$ In the particular case where the quotient is the
Calkin algebra and the lifting is to $\mathbb{B}(\mathbb{H}),$ we
proved this using different methods in \cite{LoringShulmanSpRadius},
as a partial answer to Olsen's question
\cite{OlsenCptPeturbOfOps}. 

More generally, we consider soft homogeneous relations (as defined
below) together with relations $x_{j}^{N_{j}}=0.$ In one variable,
another example of such a collection of liftable relations is
\[
\left\Vert x\right\Vert 
\leq C_{1},\quad\left\Vert x^{*}x-x^{2}\right\Vert 
\leq C_{2},\quad x^{3}
=0.
\]
In two variables, we have such curiosities as
\[
\left\Vert x\right\Vert \leq1,\quad\left\Vert y\right\Vert \leq 1,
\quad
x^{3}=0,\quad y^{3}=0,
\quad
\left\Vert x-y\right\Vert \leq\epsilon
\]
which we can now lift.

Given a $*$-polynomial in $x_{1},\dots,x_{n}$ we have the usual
relation $p(x_{1},\ldots,x_{n})=0,$ where now the $x_{j}$ are in
a $C^{*}$-algebra. In part due to the shortage of semiprojective
$C^{*}$-algebras, Blackadar \cite{Blackadar-shape-theory} suggested
that we would do well to study the relation 
$\left\Vert p(x_{1},\ldots,x_{n})\right\Vert \leq C$
for some $C>0.$ Following Exel's lead \cite{ExelSoftTorusI}, we
call this a \emph{soft polynomial relation}. Softened relations come
up naturally when trying to classifying $C^{*}$-algebras that are
inductive limits, as in \cite{elliott1996AT}, when exact relations
in the limit lead only to inexact relations in a building block in
the inductive system.

The homogeneity we need is only that there be a subset (such
as) $x_{1},\ldots,x_{r}$ of the variables and an integer $d\geq1$
so that every monomial in $p$ contains exactly $d$ factors from
$x_{1},x_{1}^{*}\dots,x_{r},x_{r}^{*}.$ 

The relation $x^{N}=0$ is ``more liftable'' than most liftable relations
in that it can be added to many liftable sets while maintaining liftability.
Other relations that behave this way are $x^{*}=x$ and $x\geq0.$
We explored semiagebraic sets (as NC topological spaces) in positive
and Hermitian variables in \cite{LoringShulmanSemialg}. 

There are still other relations that are ``more liftable'' in this
sense. We consider in this note $xyx^{*}=0$ and $xy=0.$ This is
not the end of the story, as we might have a rare case of too little
theory and too many examples.

We use many technical results from our previous work
\cite{LoringShulmanSemialg}. We also have use for the Kasparov
Technical Theorem. Indeed we use only a simplified version, but the
fully technical version can probably be used to find even more lifting
theorems in this realm. For a reference, a choice could be made from
\cite{ELP-pushBusby,Loring-lifting-perturbing,OlsenPedLifting}.

We will use the notation $a\ll b$ to mean $b$ acts like unit on
$a,$ i.e. $ab=a=ba.$

A trick we use repeatedly is to replace a single element $c$ so that
$0\leq c\leq1$ and 
\[
x_{j}c=x_{j},\quad cy_{k}=0
\]
for some sequences $x_{j}$ and $y_{k}$ with two elements $a$ and
$b$ with
\begin{equation}
0\leq a\ll b\leq1
\label{eq:KTT1}
\end{equation}
and
\begin{equation}
x_{j}a=x_{j},\quad by_{k}=0.
\label{eq:KTT2}
\end{equation}
These are found with basic functional calculus. The simplified version
of Kasparov's technical theorem we need can be stated as follows:
for $x_{1},x_{2},\cdots$ and $y_{1},y_{2},\cdots$ in a corona algebra
$C(A)=M(A)/A$ (for $A$ $\sigma$-unital) with $x_{j}y_{k}=0$ for
all $j$ and $k,$ there are elements $a$ and $b$ in $C(A)$ satisfying
(\ref{eq:KTT1}) and (\ref{eq:KTT2}).

\section{Lifting Nilpotents while Preserving Various Norms}

\begin{lem}
\label{lem:nil-lift-trick} Suppose $A$ is $\sigma$-unital $C^{*}$-algebra,
$n$ is at least $2,$ and consider the quotient
map $\pi:M(A)\rightarrow M(A)/A.$
\begin{enumerate}
\item If $x$ is an element of $M(A)$ so that $\pi\left(x^{n}\right)=0$
then there are elements $p_{1},\ldots,p_{n-1}$ and
$q_{1},\ldots,q_{n-1}$ of $M(A)$ with 
\[
j>k\implies p_{j}q_{k}=0
\]
and
\[
\pi\left(\sum_{j=1}^{n-1}q_{j}xp_{j}\right)=\pi(x).
\]
\item If $\pi(\tilde{x})=\pi(x)$ and we set
\[
\bar{x}=\sum_{j=1}^{n-1}q_{j}\tilde{x}p_{j},
\]
then $\pi(\bar{x})=\pi(x)$ and $\bar{x}^{n}=0.$
\end{enumerate}
\end{lem}

\begin{proof}
This is the essential framework that assists the lifting of nilpotents,
going back to \cite{OlsenPedLifting}. Other than a change of notation,
this is an amalgam of Lemmas 1.1, 8.1.3, 12.1.3 and 12.1.4 of
\cite{Loring-lifting-perturbing}.
\end{proof}

\begin{thm}
If $x$ is an element of a $C^{*}$-algebra $A,$ and $I$ is an ideal
and $\pi:A\rightarrow A/I$ is the quotient map, then for any natural
number $N,$ there is an element $\bar{x}$ in $A$ so 
that $\pi(\bar{x})=\pi(x)$
and
\[
\left\Vert \bar{x}^{n}\right\Vert 
= \left\Vert \pi\left(x^{n}\right)\right\Vert,
\quad (n=1,\ldots,N).
\]
\end{thm}

\begin{proof}
If $\pi\left(x^{N}\right)\neq0,$ then this is the first statement
in Theorem 3.8 of \cite{AkePederIdealPerturb}.

Assume then that $\pi\left(x^{N}\right)=0.$ Standard reductions
(Theorem~10.1.9 of \cite{Loring-lifting-perturbing}) allow us to
assume $A=M(E)$ and $I=E$ for some separable $C^{*}$-algebra $E.$
The first part of Lemma~\ref{lem:nil-lift-trick} provides elements
$p_{1},\ldots,p_{N-1}$ and $q_{1},\ldots,q_{N-1}$ in $M(E)$ with 
\[
j>k\implies p_{j}q_{k}=0
\]
and
\[
\pi\left(\sum_{j=1}^{N-1}q_{j}xp_{j}\right)=\pi(x).
\]
Let $C_{n}=\left\Vert \pi\left(x^{n}\right)\right\Vert .$
Each norm condition 
\[
\left\Vert \left(\sum_{j=1}^{N-1}q_{j}\tilde{x}p_{j}\right)^{n}\right\Vert 
\leq C_{n}
\quad(n=1,\ldots,N-1)
\]
is a norm-restriction of a NC polynomial that is homogeneous in $\tilde{x}.$
We can apply Theorem~3.2 of \cite{LoringShulmanSemialg} to find
$\hat{x}$ in $M(E)$ with
$\pi\left(\hat{x}\right)=\pi\left(\tilde{x}\right)$
and
\[
\left\Vert \left(\sum_{j=1}^{N-1}q_{j}\hat{x}p_{j}\right)^{n}\right\Vert \leq C_{n}\quad(n=1,\ldots,N-1).
\]
Since $\pi\left(\hat{x}\right)=\pi\left(x\right)$ we may apply the
second part of Lemma~\ref{lem:nil-lift-trick} to conclude that
\[
\bar{x}=\sum_{j=1}^{N-1}q_{j}\hat{x}p_{j}
\]
is a lift of $\pi(x),$ is nilpotent of order $N,$ and
\[
\left\Vert \bar{x}^{n}\right\Vert \leq C_{n}=\left\Vert \pi\left(x^{n}\right)\right\Vert 
\]
for $n=1,\ldots,N-1.$
\end{proof}

There was nothing special about the homogeneous $*$-polynomials $x^{n},$
and we can deal with more than one nilpotent variable $x$ at a time.
We say a $*$-polynomial is \emph{homogeneous of degree $r$ for some
subset} $S$ of the variables when the total number of times either
$x$ or $x^{*}$ for $x\in S$ appears in each monomial is $r.$ 
Staying consistent with the notation in \cite{LoringShulmanSemialg},
we use 
\[
p\left(\mathbf{x},\mathbf{y}\right) 
= p\left(x_{1},\ldots x_{r},y_{1},y_{2},\ldots\right)
\]
as so keep to the left the variables in subset where there is homogeneity.

\begin{thm}
Suppose $p_{1},\ldots,p_{J}$ are NC $*$-polynomials in infinitely
many variables that are homogeneous in the set of the first $r$ variables,
each with degree of homogeneity $d_{j}$ at least one. Suppose $C_{j}>0$
are real constants and $N_{k}\geq2$ are integer constants, $k=1,\ldots,r.$
For every $C^{*}$-algebra $A$ and $I\vartriangleleft A$ an ideal,
given $x_{1},\ldots,x_{r}$ and $y_{1},y_{2},\ldots$ in $A$ with
\[
\left(\pi\left(x_{k}\right)\right)^{N_{k}}=0
\]
and
\[
\left\Vert p_{j}\left(\pi\left(\mathbf{x},\mathbf{y}\right)\right)\right\Vert \leq C_{j},
\]
there are $z_{1},\ldots,z_{r}$ in $A$ with  $\pi\left(\mathbf{z}\right)=\pi(\mathbf{x})$
and 
\[
 z_{k} ^{N_{k}}=0
\]
 and
\[
\left\Vert p_{j}\left(\mathbf{z},\mathbf{y}\right)\right\Vert \leq C_{j}.
\]
\end{thm}

\begin{proof}
Again we use standard reductions to assume $A=M(E)$ and $I=E$ for
some separable $C^{*}$-algebra $E.$ Now we apply Lemma~\ref{lem:nil-lift-trick}
to each $x_{k}$ and find $p_{k,1},\ldots,p_{k,N_{k}-1}$ and
$q_{k,1},\ldots,q_{k,N_{k}-1}$ in $M(E)$ with 
\[
b>c\implies p_{k,b}q_{k,c}=0
\]
and
\[
\pi\left(\sum_{b=1}^{N_{k}-1}q_{k,b}x_{k}p_{k,b}\right)=\pi\left(x_{k}\right).
\]
We know that any $\tilde{\mathbf{x}}$
we take with
$\pi\left(\tilde{\mathbf{x}}\right)=\pi\left(\mathbf{x}\right)$
will give us
\[
\pi\left(\sum_{b=1}^{N_{k}-1}q_{k,b}\tilde{x}_{k}p_{k,b}\right)
=\pi\left(x_{k}\right)
\]
and
\[
\left(\sum_{b=1}^{N_{k}-1}q_{k,b}\tilde{x}_{k}p_{k,b}\right)^{N_{k}}=0,
\]
so we need only fix the relations
\[
\left\Vert p_{j}\left(\sum_{b=1}^{N_{1}-1}q_{1,b}\tilde{x}_{1}p_{1,b},\ldots\sum_{b=1}^{N_{r}-1}q_{r,b}\tilde{x}_{r}p_{r,b},\mathbf{y}\right)\right\Vert \leq C_{j}.
\]
These are homogeneous in
$\left\{ \tilde{x}_{1},\ldots,\tilde{x}_{r}\right\} $ so we
are done, by Theorem~3.2 of \cite{LoringShulmanSemialg}.
\end{proof}

We could add various relations on the variables $y_{1},y_{2},\dots,$
and include in the $p_{j}$ $*$-polynomials, in various ways that
ensure that there is an associated universal $C^{*}$-algebra which
is then projective. For example, we could zero them out the extra
variables (so just omit them) and impose a soft relation know for
imply all the $x_{j}$ are contractions. Let us give one specific
class of examples.

\begin{example}
We have projectivity for the universal $C^{*}$-algebra on $x_{1},\dots,x_{n}$
subject to the relations
\[
x^{N_{k}}=0,\ \left\Vert \sum x_kx_k^* \right\Vert \leq1,\ \left\Vert p_{j}\left(x_{1},\dots,x_{n}\right)\right\Vert \leq C_{j}
\]
for $C_{j}>0$ and the $p_{j}$ all NC $*$-polynomials that are homogeneous
in $x_{1},\dots,x_{n}.$
\end{example}

\section{The Relation $xyx^{*}=0.$}

We now explore setting $xyx^{*}$ to zero. This word is unshrinkable,
in the sense of \cite{tapperEmbedding}. We show that many sets of
relations involving $xyx^{*}=0$ are liftable. One example, chosen
essentially at random, is the set consisting of the relations
\[
\left\Vert x\right\Vert \leq1,\quad\left\Vert y\right\Vert \leq1,
\quad
\left\Vert xy+yx\right\Vert \leq1,\quad xyx^{*}=0.
\]

\begin{lem}
\label{lem:conditioningxyxstar} 
Suppose $A$ is $\sigma$-unital and $C(A)=M(A)/A.$ If $x$ and $y$
are elements of $M(A)$ so that $xyx^{*}=0,$ then there are elements 
\[
0\leq e\ll f\ll g\le1
\]
so that
\[
x(1-g)=x
\]
and
\[
ey+(1-e)yf=y.
\]
\end{lem}

\begin{proof}
We apply Kasparov's technical theorem to the product
$x\left(yx^{*}\right)=0$ to find 
\[
0\leq d\le1
\]
 in $C(A)$ with 
\begin{align}
xd & =x,\label{eq:xd}\\
dyx^{*} & =0.\label{eq:dyxstar}
\end{align}
We rewrite (\ref{eq:xd}) as
\begin{equation}
(1-d)x^{*}=0\label{eq:(1-d)xstar}
\end{equation}
and apply Kasparov's technical theorem to (\ref{eq:dyxstar}) and
(\ref{eq:(1-d)xstar}) to find
\[
0\leq f\ll g\leq1
\]
 in $C(E)$ with 
 \begin{align}
(1-d)f & =(1-d)\nonumber \\
dyf & =dy\label{eq:dyf}\\
gx^{*} & =0.\nonumber 
\end{align}
Thus we have $xg=0$ and
\[
0\leq1-d\ll f\ll g\leq1.
\]
We are done, with $e=1-d,$ since (\ref{eq:dyf}) gives us
\[
ey+(1-e)yf=(1-d)y+dyf=y.
\]
\end{proof}

\begin{lem}
\label{lem:xyxstar-lift-trick} 
Suppose $A$ is $\sigma$-unital and
consider the quotient map $\pi:M(A)\rightarrow M(A)/A.$
\begin{enumerate}
\item If $x$ and $y$ are elements of $M(A)$ so that $\pi\left(xyx^{*}\right)=0,$
then there are elements $e,$ $f$ and $g$ in $M(A)$ with 
\begin{equation}
0\leq e\ll f\ll g\le1,
\label{eq:e_ll_f_ll_g}
\end{equation}
\[
\pi\left(x(1-g)\right)=\pi(x)
\]
and
\[
\pi\left(ey+(1-e)yf\right)=\pi(y).
\]
\item If $\pi(\tilde{x})=\pi(x)$ and $\pi(\tilde{y})=\pi(y)$ then, if
we set
\[
\bar{x}=\tilde{x}(1-g)
\]
and
\[
\bar{y}=e\tilde{y}+(1-e)\tilde{y}f,
\]
we have $\pi(\bar{x})=\pi(x),$ $\pi(\bar{y})=\pi(y)$ 
and $\bar{x}\bar{y}\bar{x}^{*}=0.$ 
\end{enumerate}
\end{lem}

\begin{proof}
In $C(A),$ the product $\pi(x)\pi(y)\pi(x)^{*}$ is zero,
so Lemma~\ref{lem:conditioningxyxstar}
produces $e_{0},$ $f_{0}$ and $g_{0}$ in $C(A)$ with
\[
0\leq e_{0}\ll f_{0}\ll g\le1,
\]
\[
\pi(x)(1-g_{0})=\pi(x)
\]
and
\[
e_{0}\pi(y)+(1-e_{0})\pi(y)f_{0}=\pi(y).
\]
Lemma 1.1.1 of \cite{Loring-lifting-perturbing} tells us there are
lifts $e,$ $f$ and $g$ in $M(A)$ of $e_{0},$ $f_{0}$ and $g_{0}$
satisfying (\ref{eq:e_ll_f_ll_g}). Then
\[
\pi\left(x(1-g)\right)=\pi\left(x\right)\left(1-g_{0}\right)
=\pi\left(x\right)
\]
and
\[
\pi\left(ey+(1-e)yf\right)=e_{0}\pi(y)+(1-e_{0})\pi(y)f_{0}=\pi(y).
\]

As for the second statement,
\[
\pi(\bar{x})=\pi\left(\tilde{x}(1-g)\right)
=\pi(x)(1-g_{0})
=\pi(x),
\]
\[
\pi(\bar{y})=\pi(e\tilde{y}+(1-e)\tilde{y}f)
=e_{0}\pi(y)+(1-e_{0})\pi(y)f_{0}
=\pi(y)
\]
 and
\[
\bar{x}\bar{y}\bar{x}^{*}
=\tilde{x}(1-g)e\tilde{y}(1-g)\tilde{x}^{*}+\tilde{x}(1-g)(1-e)\tilde{y}f(1-g)\tilde{x}^{*}
=0
\]
since $(1-g)e=0$ and $(1-g)f=0.$
\end{proof}

\begin{thm}
\label{thm:xyx_star} 
Suppose $p_{1},\ldots,p_{J}$ are NC $*$-polynomials
in infinitely many variables that are homogeneous in the set of the
first $2r$ variables, each with degree of homogeneity $d_{j}$ at
least one. Suppose $C_{j}>0$ are real constants and $N_{j}\geq2$
are integer constants. For every $C^{*}$-algebra $A$ and 
$I\vartriangleleft A$ an ideal, given $x_{1},\ldots,x_{r}$ and
$y_{1},\ldots,y_{r}$ and $z_{1},z_{2},\ldots$ in $A$ with 
\[
\pi\left(x_{k}\right)\pi\left(y_{k}\right)\pi\left(x_{k}\right)^{*}=0,\quad(k=1,\ldots,r)
\]
and
\[
\left\Vert p_{j}\left(\pi\left(\mathbf{x},\mathbf{y},\mathbf{z}\right)\right)\right\Vert \leq C_{j},\quad(j=1,\ldots,J)
\]
there are $\bar{x}_{1},\ldots,\bar{x}_{r}$ and
$\bar{y}_{1},\ldots,\bar{y}_{r}$ in $A$ with
$\pi\left(\bar{\mathbf{x}}\right)=\pi(\mathbf{x})$ and
$\pi\left(\bar{\mathbf{y}}\right)=\pi(\mathbf{y})$ and 
\[
\bar{x}_{k}\bar{y}_{k}\bar{x}_{k}^{*}=0,\quad(k=1,\ldots,r)
\]
and
\[
\left\Vert p_{j}\left(\bar{\mathbf{x}},\bar{\mathbf{y}},\bar{\mathbf{z}}\right)\right\Vert 
\leq C_{j},
\quad(j=1,\ldots,J).
\]
\end{thm}

\begin{proof}
Without loss of generality, assume $A=M(E)$ and $I=E$ for some separable
$C^{*}$-algebra $E.$ Now we apply Lemma~\ref{lem:xyxstar-lift-trick}
to each pair $x_{j}$ and $y_{j}$ and find $e_{j},$ $f_{j}$ and
$g_{j}$ in $M(E)$ so that, given any lifts $\tilde{x}_{j}$ and
$\tilde{y}_{j}$ of $\pi(x_{j})$ and $\pi(y_{j}),$ setting
\[
\bar{x}_{j}=\tilde{x}_{j}(1-g_{j})
\]
and
\[
\bar{y}_{j}=e_{j}\tilde{y}_{j}+(1-e_{j})\tilde{y}_{j}f_{j}
\]
produces again lifts of the $\pi(x_{j})$ and $\pi(y_{j})$ with
\[
\bar{x}_{j}\bar{y}_{j}\bar{x}_{j}^{*}=0.
\]
The needed norm conditions
\[
\left\Vert
  p_{j} \left(
    \tilde{x}_{1}(1-g_{1}),\ldots,\tilde{x}_{r}(1-g_{r}),
    e_{1}\tilde{y}_{1}+(1-e_{1})\tilde{y}_{1}f_{1},
    \ldots,
    e_{r}\tilde{y}_{r}+(1-e_{r})\tilde{y}_{r}f_{r},
    \bar{\mathbf{z}}
    \right)
\right\Vert 
\leq C_{j}
\]
involve NC $*$-polynomials that are homogeneous in
$\{x_{1},\ldots,x_{r},y_{1},\ldots,y_{r}\},$
so Theorem~3.2 of \cite{LoringShulmanSemialg}
again finishes the job. 
\end{proof}

\begin{example}
For any $r,$ the $C^{*}$-algebra 
\[
C^{*}\left\langle
x_{1},\ldots x_{r},y_{1},\ldots,y_{r}
\left|
\begin{array}{c}
x_{j}y_{j}x_{j}^{*}=0,\\
\left\Vert \sum x_{j}x_{j}^{*}+y_{j}y_{j}^{*}\right\Vert \leq 1
\end{array}
\right.
\right\rangle 
\]
is projective. In particular, since projective implies residually
finite dimensional, if one could show that the $*$-algebra
\[
\mathbb{C}\left\langle
x_{1},\ldots x_{r},y_{1},\ldots,y_{r}
\left|
\begin{array}{c}
x_{j}y_{j}x_{j}^{*}=0,\\
\left\Vert \sum x_{j}x_{j}^{*}+y_{j}y_{j}^{*}\right\Vert \leq1
\end{array}
\right.
\right\rangle 
\]
is $C^{*}$-representable (as in \cite{popovych2010representability}),
then it would have a separating family of finite dimensional representations. 
\end{example}

\section{The Relations $x_{j}x_{k}=0.$}

We can work with variables that are ``half-orthogonal'' in that any
product $x_{j}x_{k}$ is zero. The $*$-monoid here contains only
monomials of the forms 
\[
x_{j_{1}}x_{j_{2}}^{*}\cdots x_{j_{2N-1}}x_{j_{2N}}^{*},
\ x_{j_{1}}x_{j_{2}}^{*}\cdots x_{j_{2N}}x_{j_{2N+1}}
\]
and their adjoints.

\begin{lem}
\label{lem:conditioningHalfOrthog} 
Suppose $A$ is $\sigma$-unital and $C(A)=M(A)/A.$ If $x_{1}\ldots,x_{r}$
are elements of $M(A)$ so that $x_{j}x_{k}=0$ for all $j$ and $k$ then there
are elements $0\leq f,g\le1$ so that
\[
fg=0
\]
and
\[
fx_{j}g=x_{j}
\]
for all $j.$
\end{lem}

\begin{proof}
We apply Kasparov's technical theorem to find $a$ and $b$ with 
\[
0\leq a\ll b\leq1
\]
and
\[
x_{j}a=a,\quad bx_{j}=0.
\]
Let $f=1-b$ and $g=a.$
\end{proof}

\begin{lem}
\label{lem:semiOrthog-lift-trick} 
Suppose $A$ is $\sigma$-unital and consider the quotient map
$\pi:M(A)\rightarrow M(A)/A.$
\begin{enumerate}
\item If $x_{1},\ldots,x_{r}$ are elements of $M(A)$ so
that $\pi\left(x_{j}x_{k}\right)=0$ for all $j$ and $k,$ then there
are elements $f$ and $g$ in $M(A)$ with
\begin{equation}
0\leq f,g\leq1,
\label{eq:0LEQf_gLEQ1}
\end{equation}
\begin{equation}
fg=0
\label{eq:fg}
\end{equation}
and
\[
\pi\left(fx_{j}g\right)=\pi\left(x_{j}\right).
\]
\item If $\pi(\tilde{x}_{j})=\pi(x_{j})$ then, if we set
\[
\bar{x}_{j}=f\tilde{x}_{j}g,
\]
we have $\pi(\bar{x}_{j})=\pi(x_{j})$ and 
\[
\bar{x}_{j}\bar{x}_{k}=0
\]
for all $f$ and $g.$ 
\end{enumerate}
\end{lem}

\begin{proof}
The products $\pi(x_{j})\pi(x_{k})$ are zero, so
Lemma~\ref{lem:conditioningHalfOrthog} gives us elements
$0\leq f_{0},g_{0}\leq1$ in $C(A)$ with $f_{0}g_{0}=0$ and
\[
f_{0}\pi\left(x_{j}\right)g_{0}=\pi\left(x_{j}\right).
\]
Orthogonal positive contractions lift to orthogonal positive contractions,
so there are $f$ and $g$ in $M(A)$ satisfying (\ref{eq:0LEQf_gLEQ1})
and (\ref{eq:fg}) that are lifts of $f_{0}$ and $g_{0},$ which
means 
\[
\pi\left(fx_{j}g\right)=f_{0}\pi\left(x_{j}\right)g_{0}=\pi\left(x_{j}\right).
\]
With $\bar{x}_{j}$ as indicated,
\[
\pi(\bar{x}_{j})=\pi(f\tilde{x}_{j}g)=f_{0}\pi(x_{j})g_{0}=\pi(x_{j})
\]
and
\[
\bar{x}_{j}\bar{x}_{k}=f\tilde{x}_{j}gf\tilde{x}_{k}g=0.
\]
\end{proof}

\begin{thm}
Suppose $p_{1},\ldots,p_{J}$ are NC $*$-polynomials in infinitely
many variables that are homogeneous in the set of the first $r$ variables,
each with degree of homogeneity $d_{j}$ at least one. Suppose $C_{j}>0$
are real constants. For every $C^{*}$-algebra $A$ and $I\vartriangleleft A$
an ideal, given $x_{1},\ldots,x_{r}$ and $y_{1},y_{2},\ldots$ in
$A$ with 
\[
\pi\left(x_{k}\right)\pi\left(x_{l}\right)=0,\quad(k,l=1,\ldots,r)
\]
and
\[
\left\Vert p_{j}\left(\pi\left(\mathbf{x},\mathbf{y}\right)\right)\right\Vert 
\leq C_{j},
\quad(j=1,\ldots,J)
\]
there are $\bar{x}_{1},\ldots,\bar{x}_{r}$ in $A$ with 
$\pi\left(\bar{\mathbf{x}}\right)=\pi(\mathbf{x})$
and 
\[
\bar{x}_{k}\bar{x}_{l}=0,\quad(k,l=1,\ldots,r)
\]
and
\[
\left\Vert p_{j}\left(\bar{\mathbf{x}},\bar{\mathbf{y}}\right)\right\Vert 
\leq C_{j},
\quad(j=1,\ldots,J).
\]
\end{thm}

\begin{proof}
The proof is essentially the same as that of Theorem~\ref{thm:xyx_star}.
\end{proof}


\begin{thebibliography}{10}

\bibitem{AkePederIdealPerturb}
Charles~A. Akemann and Gert~K. Pedersen.
\newblock Ideal perturbations of elements in {$C\sp*$}-algebras.
\newblock {\em Math. Scand.}, 41(1):117--139, 1977.

\bibitem{Blackadar-shape-theory}
Bruce Blackadar.
\newblock Shape theory for {$C\sp \ast$}-algebras.
\newblock {\em Math. Scand.}, 56(2):249--275, 1985.

\bibitem{efetov1999supersymmetry}
K.~Efetov.
\newblock {\em {Supersymmetry in disorder and chaos}}.
\newblock Cambridge Univ Pr, 1999.

\bibitem{ELP-pushBusby}
S{\o}ren Eilers, Terry~A. Loring, and Gert~K. Pedersen.
\newblock Morphisms of extensions of {$C\sp *$}-algebras: pushing forward the
  {B}usby invariant.
\newblock {\em Adv. Math.}, 147(1):74--109, 1999.

\bibitem{elliott1996AT}
G.A. Elliott and G.~Gong.
\newblock {On the classification of C*-algebras of real rank zero, II}.
\newblock {\em Ann. of Math.}, 144(3):497--610, 1996.

\bibitem{ExelSoftTorusI}
Ruy Exel.
\newblock The soft torus and applications to almost commuting matrices.
\newblock {\em Pacific J. Math.}, 160(2):207--217, 1993.

\bibitem{HeltonPositivstellensatz}
J.~William Helton and Scott~A. McCullough.
\newblock A {P}ositivstellensatz for non-commutative polynomials.
\newblock {\em Trans. Amer. Math. Soc.}, 356(9):3721--3737 (electronic), 2004.

\bibitem{lasserrePutinarSemiAlgebraic}
Jean~B. Lasserre and Mihai Putinar.
\newblock Positivity and optimization for semi-algebraic functions.
\newblock {\em SIAM J. Optim.}, 20(6):3364--3383, 2010.

\bibitem{Loring-lifting-perturbing}
Terry~A. Loring.
\newblock {\em Lifting solutions to perturbing problems in {$C\sp
  *$}-algebras}, volume~8 of {\em Fields Institute Monographs}.
\newblock American Mathematical Society, Providence, RI, 1997.

\bibitem{LoringShulmanSpRadius}
Terry~A. Loring and Tatiana Shulman.
\newblock Generalized spectral radius formula and {O}lsen's question.
\newblock arxiv:1007.4655, 2010.

\bibitem{LoringShulmanSemialg}
Terry~A. Loring and Tatiana Shulman.
\newblock Noncommutative semialgebraic sets and associated lifting problems.
\newblock {\em Trans. Amer. Math. Soc.}, to appear.
\newblock arxiv:0907.2618.

\bibitem{mandilara2006quantum}
A.~Mandilara, V.M. Akulin, A.V. Smilga, and L.~Viola.
\newblock {Quantum entanglement via nilpotent polynomials}.
\newblock {\em Phys. Rev. A}, 74(2):22331, 2006.

\bibitem{OlsenCptPeturbOfOps}
Catherine~L. Olsen.
\newblock Norms of compact perturbations of operators.
\newblock {\em Pacific J. Math.}, 68(1):209--228, 1977.

\bibitem{OlsenPedLifting}
Catherine~L. Olsen and Gert~K. Pedersen.
\newblock Corona {$C\sp *$}-algebras and their applications to lifting
  problems.
\newblock {\em Math. Scand.}, 64(1):63--86, 1989.

\bibitem{popovych2010representability}
Stanislav Popovych.
\newblock On {$O^*$}-representability and {$C^*$}-representability of
  {$*$}-algebras.
\newblock {\em Houston J. Math.}, 36(2):591--617, 2010.

\bibitem{Shulman_nilpotents}
Tatiana Shulman.
\newblock {Lifting of nilpotent contractions}.
\newblock {\em Bull. London Math. Soc.}, 40(6):1002--1006,
  2008.

\bibitem{tapperEmbedding}
P.~Tapper.
\newblock {Embedding {$*$}-algebras into {$C\sp *$}-algebras and {$C\sp
  *$}-ideals generated by words}.
\newblock {\em J. Operator Theory}, 41(2):351--364, 1999.

\end{thebibliography}
\end{document}